\documentclass[preprint,12pt]{article}

\usepackage{amsmath,amsthm,amssymb,latexsym}
\usepackage{doc}
\usepackage{color}
\usepackage{makeidx}
\usepackage{hyperref}
\usepackage{url}
\usepackage[dvips]{graphicx}

\newtheorem{thm}{Theorem}[section]
\newtheorem{cor}[thm]{Corollary}
\newtheorem{lemma}[thm]{Lemma}

\makeindex

\begin{document}

\begin{center}
{\LARGE{Gradation in Greyscales of Graphs}} \vspace*{.5cm}

{\sc  Natalia de Castro\footnote[1]{Dpto. de Matem\'atica Aplicada I, Universidad de Sevilla, e-mail: natalia@us.es, vizuete@us.es, rafarob@us.es.}
 Mar\'{\i}a A. Garrido-Vizuete$^1$,\footnote[2]{Corresponding author. Escuela T\'ecnica Superior de Ingenier\'{\i}a de Edificaci\'{o}n, Avda. Reina Mercedes, n. 4A, 41012 Sevilla, SPAIN.}, Rafael Robles$^1$,
  Mar\'{\i}a Trinidad Villar-Li\~n\'an\footnote[3]{Dpto. de Geometr\'{\i}a y Topolog\'{\i}a, Facultad de Matem\'aticas,
Universidad de Sevilla, vi\-llar@us.es.} \\
}

\end{center}

\begin{abstract}
In this work we present the notion of greyscale of a graph as a colouring of its vertices that uses colours from the real interval [0,1]. Any greyscale induces another colouring by assigning to each edge the non-negative difference between the colours of its vertices. These edge colours are ordered in lexicographical decreasing ordering and give rise to a new element of the graph: the gradation vector. We introduce the notion of minimum gradation vector as a new invariant for the graph and give polynomial algorithms to obtain it. These algorithms also output all greyscales that produce the minimum gradation vector.
This way we tackle and solve a novel vectorial optimization problem in graphs that may produce more satisfactory solutions than those ones generated by known scalar optimization approaches.
The interest of these new concepts lies in their possible applications for solving problems of engineering, physics and applied mathematics which are modeled according to a network whose nodes have assigned numerical values of a certain parameter delimited by a range of real numbers. The objective is to minimize the differences between each node and its neighbors, ensuring that the extreme values of the interval are assigned.
\end{abstract}

\vspace*{0.4cm}

\vspace*{0.4cm}
\textbf{MSC 2010 (Primary):} 05C, 68R.    \textbf{MSC 2010 (Secondary):} 05C15, 05C85,  90C47. 

\section{Introduction}

Graph colouring problems are among the most important combinatorial
optimization  problems in graph theory because of their wide
applicability in areas such as wiring of printed
circuits~\cite{clw-odmwi-92}, resource
allocation~\cite{wsn-radpbnugca-91}, frequency assignment
problem~\cite{ahkms-mp-07,gk-glvws-09,omgh-bsgc-16}, a wide
variety of scheduling problems~\cite{m-gcpas-04} or computer
register allocation~\cite{cacchm-rac-81}.

A variety of combinatorial optimization problems on graphs can be
formulated  similarly in the following way. Given a graph $G(V,E)$ and a mapping $f:V\longrightarrow \mathbb{Z}$, a new mapping  $\widehat f:E\longrightarrow
\mathbb{Z}$ is induced by $f$ such that $\widehat f(e)=|f(u)-f(v)|$ for every $e=\{u,v\}\in
E$. Then an optimization problem is formulated from several key
elements: mappings $f$ belonging to a mapping subset $S$, the
image of $V$ by $f$ and the image of $E$ by $\widehat f$. In
particular, the classic graph colouring problem, that is, colouring
the vertices of $G$ with as few colours as possible so that adjacent
vertices always have different colours, can be stated in these terms
as follows:
$$\chi(G)=\min_{f\in S} |f(V)|\ \ \text{where}\ \ S=\{ f:V\rightarrow \mathbb{Z} \text{ such that } 0\notin \widehat f(E)\}. $$
It is well known that this minimum number $\chi(G)$ is called the
chromatic  number of the graph $G$ and that its computing is an
NP-hard problem~\cite{k-rcp-72}.

It must be emphasized that the classic graph colouring problem bears in mind
the number of colours used but not what they are.
However, there are some works related to map colouring for which the
nature of the colours is essential, whereas the number of them is
fixed. The {\it maximum differential graph colouring
problem}~\cite{hkv-mdgc-11}, or equivalently the {\it antibandwidth
problem}~\cite{lvw-svbmp-84}, colours the vertices of the graph in
order to maximize the smallest colour difference between adjacent
vertices and using all the colours $1,2,\ldots,|V|$. Under the above
formulation, these problems are posed as follows:
$$\max_{f\in S}\min \widehat f(E)\ \ \text{ for } S=\{ f:V\rightarrow \mathbb{Z} \text{ such that }  f(V)=\{1,2,\dots,|V|\}\},$$
and therefore the complementary optimization case, the {\it bandwidth problem}, is given by
$$\min_{f\in S}\max \widehat f(E)\ \ \text{ for } S=\{ f:V\rightarrow \mathbb{Z} \text{ such that }  f(V)=\{1,2,\dots,|V|\}\}.$$
Note that these problems are concerned with mappings  that take values within a discrete
spectrum and with the optimization of a
scalar function. Dillencourt et al.~\cite{deg-ccggcse-07} studied a
variation of the differential graph colouring problem under the
assumption that all colours in the colour spectrum are available.
This makes the problem continuous rather than discrete. The well-known frequency assignment problem is continuous by nature, although its treatment has traditionally been discrete. However, recent works~\cite{llsh-fn-13} propose to shift the paradigm from discrete channel allocation to continuous frequency allocation. On the other hand, the key issue for process scheduling problems concerns the time representation, and in order to address real limitations, methods based on continuous-time representations have attracted a great amount of attention and provide great potential for the development of more accurate and efficient modeling and solution approaches~\cite{fl-cvdascp-04, w-sf-14}.

In this line, this paper deals with mappings taking values within the continuous spectrum $[0,1]$, where $0$ and $1$ correspond to white and black colours, respectively, and the rest of the intermediate values are grey tones. Formally, given a graph $G(V,E)$, a \textit{greyscale} $f$ of $G$ is a mapping $f:V\longrightarrow [0,1]$ such that the white and black colours are reached by $f$. Every greyscale induces a mapping  $\widehat f$ on $E$ by assigning to each edge the non-negative difference between the values of $f$ on its vertices. Whether the values of $\widehat f$ are sorted by decreasing order, the \textit{gradation vector} $grad(G,f)$ is obtained and the optimization problem of finding the minimum one, following the lexicographical order, among all the gradation vectors of greyscales of $G$ arises in a natural way. Analogously, the notion of contrast is associated to increasing order and the maximum vector, and it is widely studied in work~\cite{cgrv-cgg-17} by the same authors of this paper.

The bandwidth and antibandwidth problems are interested in optimizing the extreme colours of the edges of the graph, whereas Dillencourt et al.~\cite{deg-ccggcse-07} focus on maximizing the sum of the colours of all the edges. Under this last approach, other papers work with different sum functions (for instance, see~\cite{lll-mscp-17}). Nonetheless, both cases, extreme values and sum functions, deal with scalar objective functions. The notion of gradation vector leads us to a vectorial objective function which allocates grey tones in a manner which is both local and global: local due to the fact that the colour of every particular edge belongs to the gradation vector, and global because all edges of the graph participate in the vectorial objective function. Figure~\ref{fig:vecgradation} visually displays an example of the goodness of minimum gradation vectors versus scalar optimization. Every vertex has been associated to a \textit{big pixel} which is coloured with its grey tone and these \textit{big pixels} are next to each other according to the adjacencies between vertices (in this construction, the notion of dual graph is underlying but without considering the external face). The idea of gradation is clearly better illustrated in Figure~\ref{fig:vecgradation}(a).

\begin{figure}[htb]
 \begin{center}
\includegraphics[width=9cm]{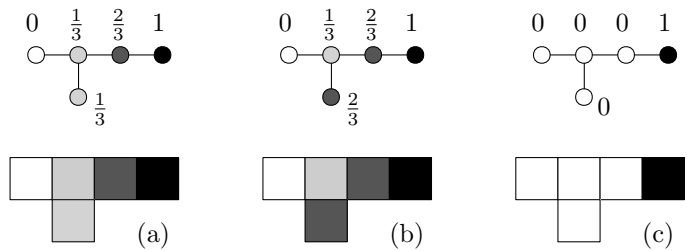}
\caption{Comparison of minimum greyscales according to different criteria about the colours of edges: (a) minimum gradation vector, (b) minimax and (c) minisum.}
\label{fig:vecgradation}
\end{center}
\end{figure}

Thus, under the formulation given above and following the lexicographical order, the minimum gradation problem on graphs is stated in the following terms:
$$\min_{f\in S}grad(G,f)\ \ \text{ for } S=\{ f:V\rightarrow [0,1] \text{ such that }  \{0,1\}\subset f(V)\}.$$

At present, although this problem seems to be a natural colouring question, we have not found the minimum gradation problem studied in these terms in the literature.

Even though it is not the main goal of this paper to  deeply focuss on possible applications of gradation in graphs,  we think that the minimum gradation greyscale could contribute to  interesting progresses on several  problems concerning real networks. Next we present two of them.

For a wide variety of systems in different areas such as biological, social, technological, and information networks the community detection problem has become extremely useful. This problem consists of identifying special groups of vertices in a graph with high concentrations of edges within such vertices and low concentrations between these groups.
This feature of real networks is called community structure \cite{Girvan}, or clustering in graphs.  For an extensive report on this topic see \cite{Santo Fortunato}. The minimum gradation  greyscale solution  could be interpreted as a new similarity measure quantifying some type of affinity between node pairs. This mapping locally minimizes the differences between each vertex and its neighbours and takes into account the global  distances in the network.  Besides, extremal values would be assigned to antipodal vertices.

It is well known that graph theory  is used to modelize many kinds of  networks services.
The minimum gradation  greyscale solution would model a possible almost uniform distribution of a service through a network, from the sources to the sinks. For  water  supply networks, problems such as  minimizing the amount of dissipated power in the water network and establishing pressure control techniques, among others, are  studied. In \cite{DiNardo2013} the method of graph partitioning is proposed to  solve them. We  guess that  for a given water network graph model, the minimum gradation  greyscale solution  would help and complement the design of a good distribution of the water through the network in such a manner that  the water or pressure losses  between contiguous pipes would be minimized  at the same time. A source or sink vertex can be modeliced with extreme values prefixed in the greyscale. Thus, some of the vertices have preassigned grey tones and the aim is to obtain the minimum gradation vector preserving these fixed grey tones.

The outline of the paper is as follows: Section~\ref{sec:prelimi} formally introduces the notion of gradation on
graphs through concepts such as greyscale and gradation vector, and formulates the two problems for study: \textit{minimum gradation} and  \textit{restricted minimum gradation on graphs}. In Section~\ref{sec:gradation}, several results about the nature of the gradation problems are first established, and then the polynomial nature of both problems is proved by designing algorithms that provide minimum gradation vectors and all greyscales that give rise to them. Finally, in Section~\ref{sec:openp} we conclude with some remarks and highlight some open problems. 

\section{Preliminaries}
\label{sec:prelimi}

This section is devoted to establishing the necessary definitions about
gradation on graphs  and to formulating the problems to be studied in this
paper. Since the gradation and contrast notions together arise in
greyscales of graphs, the basic concepts about contrast are also
presented. Given a graph\footnote{Graphs in this paper are finite,
undirected and simple and are denoted by $G(V,E)$, where $V$ and $E$
are its vertex-set and edge-set, respectively. The number of
elements of $V$ and $E$ are denoted by $n$ and $m$, respectively.
For further terminology we follow~\cite{h-gt-90}.} $G(V,E)$, a
\textit{greyscale} $f$ of $G$ is a mapping on $V$ to the interval
$[0,1]$ such that $f^{-1}(0)\neq \emptyset$ and $f^{-1}(1)\neq
\emptyset$. For each vertex $v$ of $G$, we call $f(v)$ the
\textit{grey tone} of $v$, or more generally, the \textit{colour} of
$v$, and notice that two adjacent vertices can have mapped the same
grey tone. In particular, values $0$ and $1$ are called
\textit{the extreme tones}, that is, \textit{white and black
colours}, respectively. In a natural way, the notion of \textit{complementary greyscale}
arises  for each greyscale $f$ such that it maps every vertex $v$ of
$G$ to $1-f(v)$.

Associated to each greyscale $f$ of the graph $G(V,E)$, the mapping
$\widehat{f}: E \rightarrow [0,1]$ is defined as
$\widehat{f}(e)=|f(u)-f(v)|$ for every $e=\{u,v\}\in E$ and
represents the gap or increase between the grey tones of vertices $u$ and $v$. The value $\widehat{f}(e)$ is also said to be
the \textit{grey tone} of edge $e$. Thus, we deal with \textit{coloured
vertices and edges} by $f$ and  $\widehat{f}$, respectively. Note
that the same mapping $\widehat{f}$ associated to
the greyscale $f$ and its complementary one is obtained.

The \textit{gradation vector} and the \textit{contrast vector}
associated  to the greyscale $f$ of $G$ are vectors $grad(G,f) =
(\widehat{f}(e_m), \widehat{f}(e_{m-1}) , \ldots , \widehat{f}(e_1)
)$ and $cont(G,f) =(\widehat{f}(e_1), \widehat{f}(e_2), \ldots,
\widehat{f}(e_m))$, respectively, where the edges of $G$ are indexed
such that $\widehat{f}(e_i) \leq \widehat{f}(e_j)$ whether $i < j$,
that is, in ascending order of their grey tones. Thus, it can be noticed that
the components of any contrast vector are ordered in ascending order
and those of any gradation vector in decreasing order. For the sake
of clarity and when the graph is fixed, the gradation and contrast
vectors associated to a greyscale $f$ will be denoted by
${\mathcal{G}}_f$  and ${\mathcal{C}}_f$, respectively. Figure~\ref{fig:greyscaleK4} shows two greyscales of the graph $K_4$, $f$ and $f'$,  whose corresponding  gradation vectors are  ${\mathcal{G}}_f=(1, \frac{1}{2}, \frac{1}{2},  \frac{1}{2}, \frac{1}{2},0)$ and ${\mathcal{G}}_{f'}=(1,\frac{2}{3}, \frac{2}{3},\frac{1}{3}, \frac{1}{3}, \frac{1}{3} )$, respectively.

Given two greyscales $f$ and $f'$ of a graph $G$, we say that $f$
has  \textit{better gradation} than $f'$ if the gradation vector
${\mathcal{G}}_f$ is smaller than ${\mathcal{G}}_{f'}$ following the
lexicographical order, that is, ${\mathcal{G}}_f <
{\mathcal{G}}_{f'}$. Thus, the descending order of gradation vectors
determines the goodness in terms of gradation. Then, $f$ is said to
be \textit{smaller or greater by gradation} than $f'$ if
${\mathcal{G}}_f < {\mathcal{G}}_{f'}$ or ${\mathcal{G}}_f >
{\mathcal{G}}_{f'}$, respectively. In a similar way, we say that
$f$ has \textit{better contrast} than $f'$ if the contrast vector
${\mathcal{C}}_f$ is greater than ${\mathcal{C}}_{f'}$ following the
lexicographical order, that is, ${\mathcal{C}}_f >
{\mathcal{C}}_{f'}$. Thus, the ascending order of contrast vectors
determines the goodness in terms of contrast. Then, $f$ is said to
be \textit{smaller or greater by contrast} than $f'$ if
${\mathcal{C}}_f < {\mathcal{C}}_{f'}$ or ${\mathcal{C}}_f >
{\mathcal{C}}_{f'}$, respectively.  The greyscale $f$ of
Figure~\ref{fig:greyscaleK4} is greater than $f'$ by both contrast
and gradation and so $f$ has better contrast than $f'$ but $f'$ has
better gradation than $f$.

\begin{figure}
\begin{center}
\includegraphics[width=7cm]{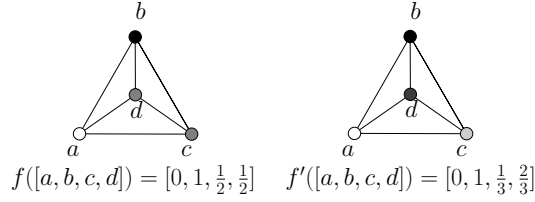}
\caption{Two greyscales $f$ and $f'$ of the graph $K_4$.}
\label{fig:greyscaleK4}
\end{center}
\end{figure}

Thus, the problems of finding the minimum gradation vector and the
maximum contrast vector  naturally arise in the contexts of
gradation and contrast on graphs. From now on, we focus on gradation
problem and point out \cite{cgrv-cgg-17} for the contrast problem.

Then, a greyscale of a graph $G$ whose gradation vector is minimum
is called a  \textit{minimum gradation greyscale} of $G$ and the
following problem is formulated:

\textbf{Minimum gradation on graphs} (\textsc{migg}): given a connected  graph $G(V,E)$, finding the minimum gradation vector and all the minimum gradation greyscales.

$$\min_{f\in S}grad(G,f)\ \ \text{ for } S=\{ f:V\rightarrow [0,1] \text{ such that }  \{0,1\}\subset f(V)\}.$$

We deal with the restricted version of this problem when the grey
tones of some  vertices are a priori known and the aim is to obtain
the minimum gradation vector preserving the fixed grey tones. This
situation leads to the concept of incomplete greyscale. Given a
graph $G(V,E)$ and a nonempty proper subset $V_c$ of $V$, \textit{an
incomplete $V_c$-greyscale} of $G$ is a mapping on $V_c$ to the
interval $[0,1]$. Note that this incompleteness means both that the
mapping is not defined on all the vertices of $G$ and
that the extreme tones do not necessarily belong to the range of the
incomplete mapping. A greyscale $f$ is \textit{compatible} with an
incomplete $V_c$-greyscale $g$ if $f(u)=g(u)$ for all $u \in V_c$.

\textbf{Restricted minimum gradation on graphs} (\textsc{rmigg}):
given a  connected graph $G(V,E)$ and an incomplete $V_c$-greyscale
$g$ of $G$, finding the gradation vector that is minimum among all
the gradation vectors of greyscales compatible with $g$, as well as
determining all these greyscales.

$$\min_{f\in S}grad(G,f)\ \ \text{ for } S=S_1 \cap S_2$$
$$\text{where } S_1=\{ f:V\rightarrow [0,1] \ / \  \{0,1\}\subset f(V) \} \text{ and }$$
$$S_2=\{f:V\rightarrow [0,1] \ / \  f \text{ is compatible with an incomplete greyscale given of $G$}\}.$$

Resolving each of these problems means finding the appropriate
minimum gradation  vector and all their minimum gradation greyscales
except the complementary ones. Note that the \textsc{migg} and
\textsc{rmigg} problems are posed for connected graphs but general
graphs can be also considered, and in this case each connected component has to be considered separately. 

\section{Minimum gradation problem}
\label{sec:gradation}

In this section several results about the nature of the gradation problems are first established, which let us
prove the correctness of our polynomial procedures. These algorithms provide minimum gradation vectors and all greyscales that give rise to them, all minimum gradation greyscales. Some cases are distinguished in the \textsc{rmigg} problem due to the special role that the extreme tones $0$ and $1$ play in the concept of greyscale and, from a common subroutine, several polynomial algorithms are designed according to the possible existence of $0$'s or $1$'s as prefixed colours. Finally, the \textsc{migg} problem is also solved in polynomial time.

Before giving our results, we make some interesting observations about the \textsc{migg} problem. Given a greyscale of a connected graph, the components of the gradation vector are sorted by decreasing order and then, our purpose is to obtain the minimum gradation vector. Therefore, our question can be considered as a minimax problem. By other hand, throughout this section we deal with gradation vectors and the greyscales which are associated to. Notice that, given a graph, the minimum gradation vector is unique but there can exist different minimum gradation greyscales which give rise to the same minimum gradation vector (see an example in Figure~\ref{fig:numgrayscale}).

\begin{figure}[htb]
 \begin{center}
\includegraphics[width=9cm]{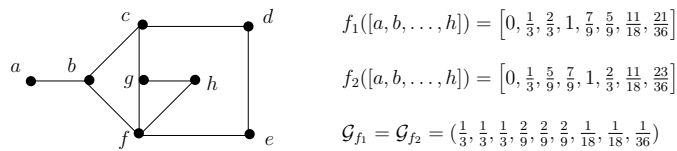}
\caption{The minimum gradation vector can be achieved from different greyscales.}
\label{fig:numgrayscale}
\end{center}
\end{figure}

About paths and distance, we follow the terminology of~\cite{h-gt-90}. Given a connected graph $G$, the \textit{distance} $d(u,v)$ between two vertices $u$ and $v$ in $G$ is the length of a shortest path joining them; a \textit{$u-v$ path} is a path joining the $u$ and $v$ vertices of $G$ and a \textit{$u-v$ geodesic} is a shortest $u-v$ path. The \textit{diameter} $d(G)$ is the length of any longest geodesic, which is called a \textit{diameter geodesic}, and two vertices $u$ and $v$ are \textit{antipodal} if $d(u,v)=d(G)$.

The two following definitions, support greyscale and edge-colour-increase mapping, provide the key tools to solve the \textsc{migg} problem for connected graphs.

Let $G=(V,E)$ be a connected graph of diameter $d(G)$ and let $u$ and $v$ be two antipodal vertices of $G$. Then, the {\em support greyscale} for $u$ and $v$ is the mapping on $V(G)$ given by
$$f\langle u,v \rangle (w)= \frac{d(w,u)-d(w,v)+d(G)}{2d(G)}.$$

It is straightforward to check that $f \langle u,v \rangle$ is a greyscale of $G$, that is, it takes values from the interval $[0,1]$ and, in particular, $f \langle u,v \rangle(u)=0$ and $f \langle u,v \rangle (v)=1$. The following lemma establishes the values of $\widehat{f} \langle u,v \rangle$.

\begin{lemma}
\label{lemma:sggorro}
Let $f\langle u,v \rangle$ be the support greyscale associated to the pair of antipodal vertices $u$ and $v$ of a connected graph $G(V,E)$ of diameter $d(G)$. Then, the only components of the gradation vector ${\mathcal{G}}_{f \langle u,v \rangle }$ are $\frac{1}{d(G)}$, $\frac{1}{2d(G)}$ and $0$.
\end{lemma}

\begin{proof}
Let $e=\{w_1,w_2\}$ be an edge of $G$. Due to the adjacency between $w_1$ and $w_2$, it holds that $d(w_2,u)=d(w_1,u)+k$, with $k \in \{-1,0,1\}$. These three values of $k$ give rise to three cases to analyze associated to $u$ (index $i$) and, analogously, three cases for $v$ (index $j$). That is, from the definition of the support greyscale and distinguishing these nine Cases $i.j$, with $i,j=1,2,3$, the following values for $|f\langle u,v\rangle(w_1)-f\langle u,v\rangle(w_2)|$ are reached:
\begin{itemize}
\item Cases 1.1, 2.2 and 3.3: $|f\langle u,v \rangle (w_1)-f \langle u,v\rangle (w_2)|=0$.

\end{itemize}
\end{proof}

In particular, the values of $\widehat{f} \langle u,v \rangle (e)$ on the edges of any $u-v$ geodesic are $\frac{1}{d(G)}$ because of the diameter notion, that is, the longest geodesic in $G$. So, at least, the first $d(G)$ components of ${\mathcal{G}}_{f \langle u,v \rangle }$ are equal to $\frac{1}{d(G)}$.

However, let us remark that support greyscales do not lead to minimum gradation vectors in general. The support greyscale for $u$ and $v$ of the graph in Figure~\ref{fig:supportnonmin} is $f \langle u,v \rangle([ u,a,b\ldots h,i,v ])=[0,\frac{1}{4},\frac{1}{4},\frac{1}{4},\frac{1}{4},\frac{1}{2},\frac{1}{2},\frac{1}{2},\frac{3}{4},1]$, and so its associated gradation vector is given by ${\mathcal{G}}_{f \langle u,v \rangle}=( \frac{1}{4}, \frac{1}{4}, \frac{1}{4}, \frac{1}{4}, \frac{1}{4}, 0,0,0,0,0 )$. Nevertheless, the minimum gradation vector of this graph of diameter $4$ is $(\frac{1}{4}, \frac{1}{4}, \frac{1}{4}, \frac{1}{4}, \frac{1}{16}, \frac{1}{16},\frac{1}{16}, \frac{1}{16}, 0 ,0)$, that can be obtained from the greyscale $$f( [ u,a,b\ldots h,i,v ])= [0,\frac{1}{4},\frac{5}{16},\frac{6}{16},\frac{7}{16},\frac{1}{2},\frac{1}{2},\frac{1}{3},\frac{3}{4},1].$$

\begin{figure}[htb]
 \begin{center}
\includegraphics[width=4cm]{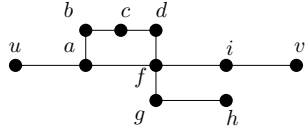}
\caption{The support greyscale for $u$ and $v$ of this graph of diameter $4$ is not a minimum gradation greyscale.}
\label{fig:supportnonmin}
\end{center}
\end{figure}
Then, given two antipodal vertices, a special greyscale has been defined which maps them to the values $0$ and $1$. In some sense, Corollary~\ref{cor:antizo} will state the reciprocal of this fact.

On the other hand, given a greyscale $f$ of a connected graph $G(V,E)$, the \textit{edge-colour-increase mapping} $F$ is defined as the mapping $F: V \times V \longrightarrow [0,1]$ such that $$F(u,v)=\left \{ \begin{tabular}{ccc} $\frac{|f(u)-f(v)|}{d(u,v)}$ & & if $u \neq v$, \\
 $0$ & & if $u=v$. \end{tabular} \right.$$
The value of $F(u,v)$ can be viewed as `the amount of colour' that every edge of any $u-v$ geodesic would be given whether the colour increase between $u$ and $v$ were fairly distributed along the $u-v$ geodesic. Throughout this section, the relationship between $\hat{f}$ and $F(u,v)$ is studied.

Geodesics play an essential role in the \textsc{migg} problem. Our next results are established on the set of geodesics of the given graph and they state several links between the values of $f$, $\hat{f}$ and $F$, according to the position of the vertices into the geodesics.

\begin{lemma}
\label{lemma:tech}
Let $f$ be a greyscale of a graph $G(V,E)$ and let $u$ and $v$ be a pair of vertices of $G$. For each vertex $w$ of each $u-v$ geodesic and different to $u$ and $v$, it holds that, $$F(u,v) < \mbox{max}\{F(u,w),F(w,v) \} \  \ \mbox{or}  \ \ F(u,v)=F(u,w)=F(w,v).$$

Moreover, the above equalities only hold whenever $f(w)$ belongs to the interval of extremes $f(u)$ and $f(v)$.
\end{lemma}

\begin{proof}
The relative position of $f(w)$ with respect to $f(u)$ and $f(v)$ determines three cases. For the sake of clarity and without loss of generality we may suppose that $f(u) \leq f(v)$. On the other hand, it is clear that $d(u,v)=d(u,w)+d(w,v)$, and since $w$ is different to $u$ and $v$, then $d(u,w) < d(u,v)$ and $d(w,v) < d(u,v)$.

\begin{enumerate}
\item If $f(w) \leq f(u) \leq f(v)$, then $$\frac{f(v)-f(u)}{d(u,v)} \leq \frac{f(v)-f(w)}{d(u,v)} < \frac{f(v)-f(w)}{d(w,v)}.$$ That is, $F(u,v) < F(w,v)$ and $F(u,v) < \mbox{max}\{F(u,w),F(w,v) \}$ trivially follows.
\item If $f(u) \leq f(v) \leq f(w)$, it proceeds as in the previous case obtaining the same inequality.
\item If $f(u) \leq f(w) \leq f(v)$, we prove that it is not possible $F(u,v) > \mbox{max}\{F(u,w),F(w,v) \}$ and if $F(u,v) = \mbox{max}\{F(u,w),F(w,v) \}$, then the three values of $F$ are equal.

    Assume to the contrary that $F(u,v) > \mbox{max}\{F(u,w),F(w,v) \}$. Hence,

    \begin{equation}
    \label{eq1}
    F(u,w) < F(u,v) \Rightarrow f(w)-f(u) < d(u,w)F(u,v)
    \end{equation}

    and
    \begin{equation}
    \label{eq2}
    F(w,v) < F(u,v) \Rightarrow f(v)-f(w) < d(v,w) F(u,v)
    \end{equation}
    Adding (\ref{eq1}) and (\ref{eq2}), we obtain
    $$f(v)-f(u) < (d(u,w)+d(w,v))F(u,v) \Rightarrow f(v)-f(u) < d(u,v)F(u,v)$$
    $$\Rightarrow F(u,v) < F(u,v),$$
    which is a contradiction.

    Next, if $F(u,v) = \mbox{max}\{F(u,w),F(w,v) \}$ we assume $F(u,v)=F(u,w)$.
    \begin{equation}
    \label{eq3}
    F(u,v)=F(u,w) \Rightarrow f(w)-f(u)=(f(v)-f(u))\frac{d(u,w)}{d(u,v)}
    \end{equation}
    Now,
    $$F(w,v)=\frac{(f(v)-f(u))-(f(w)-f(u))}{d(w,v)}$$
    and by (\ref{eq3}) and since $d(u,v)-d(u,w)=d(w,v)$,
    $$F(w,v)=\frac{(f(v)-f(u))d(w,v)}{d(u,v)d(w,v)}=F(u,v).$$
    So, $F(u,v)=F(u,w)=F(w,v)$.

    Under the assumption that $F(u,v)=F(w,v)$
    and by similar arguments, the same result is obtained.
\end{enumerate}
\end{proof}

The following two results establish connections between the mappings $\widehat{f}$ and $F$ on geodesics.

\begin{cor}
\label{cor:lowerb}
Let $f$ be a greyscale of a graph $G(V,E)$ and let $u$ and $v$ be a pair of vertices of $G$. For each $u-v$ geodesic $P_{u-v}$, it holds that $$\displaystyle F(u,v) \leq \max_{e \in P_{u-v}} \widehat{f}(e).$$
\end{cor}

\begin{proof}
We denote $P_{u-v}$ by $\{u=w_0,e_1,w_1,e_2,w_2, \dots, w_{l-1},e_l,v=w_l\}$ as alternating sequence of vertices and edges and hence $l=d(u,v)$. In order to prove the result, a stronger assertion will be stated, that is, $F(u,v) \leq \max\{ \widehat{f}(e_1),\ldots, \widehat{f}(e_i),F(w_i,v) \}$ for $i=1,\ldots,l-1$.

For $i=1$, Lemma~\ref{lemma:tech} applied to $w_1$ of $P_{u-v}$ and the fact that $\widehat{f}(e_1)=F(u,w_1)$ lead trivially to $F(u,v) \leq \max\{ \widehat{f}(e_1),F(w_1,v) \}$.

Inductively, suppose $F(u,v) \leq \max\{ \widehat{f}(e_1),\ldots, \widehat{f}(e_i),F(w_i,v) \}$. Lemma~\ref{lemma:tech} is again applied, in this case to $w_{i+1}$ as vertex of the path $\{ w_i, e_{i+1},w_{i+1},\ldots,v\}$, obtaining that $$F(w_i,v) \leq \max\{ F(w_i,w_{i+1}),F(w_{i+1},v) \} = \max\{ \widehat{f}(e_{i+1}),F(w_{i+1},v) \}.$$ Thus, the induction hypothesis and this inequality about $F(w_i,v)$ give rise to the result for $i+1$.
\end{proof}

\begin{lemma}
\label{lemma:eqFfhat}
Let $f$ be a greyscale of a graph $G(V,E)$, let $u$ and $v$ be a pair of vertices of $G$ and let $P_{u-v}$ be a $u-v$ geodesic. If $\displaystyle F(u,v) = \max_{e \in P_{u-v}} \widehat{f}(e)$ then $F(u,v)=\widehat{f}(e)$ for every edge $e \in P_{u-v}$.
\end{lemma}

\begin{proof}
For the sake of simplicity and without loss of generality, let the alternating sequence of vertices and edges $\{u=w_0,e_1,w_1,e_2,w_2, \dots, w_{l-1},e_l,v=w_l\}$ be the $u-v$ geodesic $P_{u-v}$ where $f(u) \leq f(v)$. Assume to the contrary that there exists $e_j$ an edge of $P_{u-v}$ such that $\widehat{f}(e_j) < F(u,v)$.

The following intervals are considered for $1\leq i \leq l$:

$$I_i=\left\{ \begin{tabular}{ccc} $[f(w_{i-1}),f(w_{i})]$ & \mbox{if} & $f(w_{i-1}) < f(w_{i})$ \\
$[f(w_i),f(w_{i-1})]$ & if & $f(w_{i}) < f(w_{i-1}) $  \\ $\emptyset$ & if & $f(w_{i-1})=f(w_i)$ \end{tabular}\right.$$
Thus, the union of these intervals is a cover of $[f(u),f(v)]$ and therefore the following contradiction is achieved:
$$f(v)-f(u) \leq \sum_{i=1}^l |f(w_{i-1})-f(w_i)|=\sum_{i=1}^l \widehat{f}(e_i)=\widehat{f}(e_j)+\sum_{i=1, i \neq j }^l \widehat{f}(e_i) \leq $$
$$\leq \widehat{f}(e_j)+ (l-1) \max_{1 \leq i \leq l} \widehat{f}(e_i) = \widehat{f}(e_j) +(l-1) F(u,v) <$$
$$< F(u,v) +(l-1)F(u,v) =l F(u,v)=f(v)-f(u).$$
\end{proof}

The two following results highlight the key role that the maximum value of the edge-colour-increase mapping $F$ plays in the \textsc{migg} problem.

\begin{cor}
\label{cor:maxwithv}
Let $f$ be a greyscale of a connected graph $G(V,E)$ and let $u$ and $v$ be a pair of vertices of $G$ such that $\displaystyle F(u,v)=\max_{a,b\in V} F(a,b)$ and $f(u)\leq f(v)$. For each vertex $w$ of a $u-v$ geodesic the following holds: $$f(w)=f(u)+d(u,w)F(u,v).$$
\end{cor}

\begin{proof}
If $w$ is $u$ or $v$, then the result holds trivially. Otherwise, since $\displaystyle F(u,v)=\max_{a,b\in V} F(a,b)$, in particular, it holds that $$F(u,v) \geq \mbox{max}\{F(u,w),F(w,v) \}.$$ Then, by Lemma~\ref{lemma:tech}, $F(u,v)=F(u,w)=F(w,v)$ and $f(w)$ belongs to the interval $[f(u),f(v)]$. Now, from $F(u,v)=F(u,w)$, it is trivial to obtain the required statement for $f(w)$ in the following way:

$$F(u,v)=F(u,w)=\frac{f(w)-f(u)}{d(u,w)} \Rightarrow f(w)=f(u)+d(u,w)F(u,v).$$
\end{proof}

Note that, if $f(v) \leq f(u)$, the result states by a similar proof using $F(u,v)=F(w,v)$ that $f(w)=f(v)+d(w,v)F(u,v)$.

\begin{cor}
\label{cor:maxwithe}
Let $f$ be a greyscale of a connected graph $G(V,E)$ and let $u$ and $v$ be a pair of vertices of $G$ such that $\displaystyle F(u,v)=\max_{a,b\in V} F(a,b)$. For each edge $e$ of a $u-v$ geodesic the following holds: $$\widehat{f}(e)=F(u,v).$$
\end{cor}

\begin{proof}
Without loss of generality, $f(u) \leq f(v)$ can be assumed. Let $a$ and $b$ be the vertices of the edge $e$. Corollary~\ref{cor:maxwithv} is applied to $a$ and $b$, and since $d(u,b)=d(u,a)\pm 1$, it holds that $f(a)=f(u)+d(u,a)F(u,v)$ and
$f(b)=f(u)+(d(u,a)\pm 1) F(u,v)$.

Then,  $\widehat{f}(e)=|f(b)-f(a)|=|\pm F(u,v)|=F(u,v)$.
\end{proof}

Next, a property for minimum gradation greyscales related to antipodal vertices which are white- and black-coloured is established.

\begin{cor}
\label{cor:antizo}
If $f$ is a minimum gradation greyscale  of a graph $G$ and $u$ and $v$ are vertices of $G$ such that $f(u)=0$ and $f(v)=1$, then $u$ and $v$ are antipodal vertices.
\end{cor}

\begin{proof}
By Lemma~\ref{lemma:sggorro}, the values of the colour of the edges by any support greyscale are $\frac{1}{d(G)}$, $\frac{1}{2d(G)}$ and $0$, sorted by decreasing order. Hence, $\widehat{f}(e)\leq \frac{1}{d(G)}$ for all $e \in E$ and so $\displaystyle \max_{e \in E} \widehat{f}(e) \leq \frac{1}{d(G)}$. Furthermore, since $f(u)=0$ and $f(v)=1$, $F(u,v)=\frac{1}{d(u,v)}$ and by Corollary \ref{cor:lowerb} $\displaystyle F(u,v)=\frac{1}{d(u,v)} \leq \max_{e \in E} \widehat{f}(e) \leq \frac{1}{d(G)}$. Then $d(u,v) \geq d(G)$, that is, $d(u,v)=d(G)$.
\end{proof}

The characterization of the minimum gradation vector for paths is a consequence of some of the above-proved results.

\begin{cor}
\label{cor:miggpaths}
The minimum gradation vector of the path of length $n$ is the vector whose all components are equal to $\frac{1}{n}$.
\end{cor}

\begin{proof}
Let $\{u=w_0,e_1,w_1,e_2,w_2, \dots, w_{n-1},e_n,v=w_n\}$ be the path of length $n$  and let $f$ be a greyscale such that ${\mathcal{G}_f}$ is the minimum gradation vector of the path. There must exist two vertices mapped to $0$ and $1$ by $f$, and by Corollary~\ref{cor:antizo} the only option is that they are $u$ and $v$. Thus, $F(u,v)=\frac{1}{n}$.

By other hand, it is easy to check that $\hat{f}\langle u,v \rangle (e_i)=\frac{1}{n}$ for $i=1 \ldots n$, being $f\langle u,v \rangle $ the support greyscale for $u$ and $v$. So, $\displaystyle{\max_{1 \leq i \leq n} \widehat{f}(e_i)\leq \frac{1}{n}}$ follows. Moreover, Corollary~\ref{cor:lowerb} guarantees that $F(u,v)$ is a lower bound of $\displaystyle{\max_{1 \leq i \leq n} \widehat{f}(e_i)}$.

Hence $\frac{1}{n}=F(u,v) \leq \displaystyle{\max_{1 \leq i \leq n} \widehat{f}(e_i)} \leq \frac{1}{n}$, and by Lemma~\ref{lemma:eqFfhat} we conclude that $\hat{f}(e_i)=\frac{1}{n}$ for $i=1\ldots n$.
\end{proof}

Our next aim is to design algorithms which provide all the minimum gradation greyscales of a connected graph, for both \textsc{migg} and \textsc{rmigg} problems. These greyscales are obtained in a stepwise manner by incomplete greyscales such that each of these is compatible with the previous one. Thus, an iterative procedure is carried out that is based on the operation of deleting coloured edges and isolated coloured vertices. We proceed first to devise the \textsc{$V_c$-Compatible-Complete-Mapping} common subroutine which is applied to solve both the different \textsc{rmigg} problems according to the possible existence of the extreme tones as prefixed colours, and  the \textsc{migg} problem.

\textsc{Procedure:} \textsc{$V_c$-Compatible-Complete-Mapping}

\textbf{Input:} An incomplete $V_c$-greyscale $g$ of a connected graph $G(V,E)$.

\textbf{Output:} A mapping $f$ on $V$ compatible with $g$.

\begin{enumerate}
  \item Initialize $G_1(V_1^1,E_1) \leftarrow G(V,E)$

 \item Initialize  $i \leftarrow 1$

 \item Initialize  $l(1) \leftarrow 1$

  \item \label{step:compatible} \textbf{If} $u \in V_c$ \textbf{do} $f(u)=g(u)$;

  \item \label{step:while-loop} \textbf{While} $|V_c|<|V|$ \textbf{do}

  \begin{enumerate}

   \item Compute the distance matrix $D_i$ of $G_i$;

   \item  \label{step:maximum}  Compute the finite value $M_i= \displaystyle{\max_{1\leq j \leq l(i)} }\{ F(a,b) : \{a,b\} \subseteq V_i^j\cap V_c  \}$  and the set $\displaystyle S_i=\{\{u,v\} \subseteq V_i \cap V_c  :
   F(u,v)=M_i \}$, where distances are taken from $D_i$;

   \item \label{step:maxgeo} \textbf{For} each $\{u,v\}\in S_i$ and considering distances from $D_i$ \textbf{do}
      \begin{enumerate}
        \item Compute  $A=\{w\in V_i : w \in \mbox{$u-v$ geodesic of $G_i$}\}$;
        \item \label{step:assignation} \textbf{For} each $w \in A$ \textbf{do} $$f(w)=\left\{\begin{array}{lll} f(u)+d(w,u)\ M_i & \mbox{if} & f(u)\leq f(v) \\
f(v)+d(w,v)\ M_i & \mbox{if} & f(u)>f(v)\end{array}\right. $$
\item $V_c \leftarrow V_c \cup A$
      \end{enumerate}

   \item Let $G_{i+1}(V_{i+1},E_{i+1})$ be the subgraph of $G_i(V_i,E_i)$ obtained by deleting all the edges $w_1w_2$ with $w_1,\, w_2 \in A$
    and removing the resulting isolated vertices. Let $V_{i+1}^j$ be the vertex-sets of the connected components of $G_{i+1}$ for $j=1 \ldots l(i+1)$;

   \item \label{step:terminals} \textbf{If} $S_i=\emptyset$, each set $V_i^j$ contains exactly one vertex $w^j$ in $V_c$  \textbf{then}

    \begin{enumerate}
    \item \textbf{For} $j:=1$ to $l(i)$ \textbf{do} $f(u)=f(w^j)$ with $u \in V_i^j$;
    \item $V_c \leftarrow V_c \cup V_i^1 \cup \ldots \cup V_i^{l(i)}$

    \end{enumerate}

   \item $i \leftarrow i+1$;

   \end{enumerate}
\end{enumerate}

Note that the mapping generated by \textsc{$V_c$-Compatible-Complete-Mapping} procedure and the input incomplete $V_c$-greyscale have the same range, and therefore that mapping is not necessarily a greyscale due to the possible nonexistence of the extreme tones as values reached by the mapping. It is explanatory to point out that the grey tone assigned by Step~\ref{step:assignation} can be also obtained as follows:

$$f(w)=\left\{\begin{array}{lll} f(v)-d(w,v)\ M_i & \mbox{if} & f(u)\leq f(v) \\
f(u)-d(w,u)\ M_i & \mbox{if} & f(u)>f(v)\end{array}\right. $$

Next, the \textsc{rmigg} problems are resolved according to the nature of the prefixed colours, that is, distinguishing whether or not the extreme tones are prefixed values.

\begin{thm}
\label{thm:01RMIGG} Let $g$ be an incomplete $V_c$-greyscale of a connected graph $G(V,E)$ such
that the extreme tones $0$ and $1$ are reached by $g$.  Then it is possible to obtain, in polynomial
time, a greyscale of $G$ compatible with $g$ whose gradation vector is the minimum one among all
gradation vectors of greyscales compatible with $g$.

Moreover, such greyscale is unique and it is provided by the \textsc{$V_c$-Compatible-Complete-Mapping} algorithm.
\end{thm}

\begin{proof}
In order to state the result, the \textsc{$V_c$-Compatible-Complete-Mapping} algorithm is first proved to be finite and polynomial. Then, it is shown to provide the unique greyscale compatible with $g$ such that its gradation vector is minimum among all gradation vectors of greyscales compatible with $g$.

At least one vertex is coloured at every iteration of the while-loop in Step~\ref{step:while-loop}, either by Substeps $(e)$ or $(f)$, and hence it ends after at most $|V|-|V_c|$ iterations. The time complexity of computing distance matrices (Step $3(c)$) dominates the time complexity of the rest of the steps and that can be done in $O(n^3)$ time applying the Floyd-Warshall algorithm~\cite{f-a97sp-62,w-tbm-62}. So, the while-loop and the Step $3(c)$ determine the polynomial time of the \textsc{$V_c$-Compatible-Complete-Mapping} algorithm, that is, $O(n^4)$ time.

Now, the algorithm output $f$ is proved to be a well-defined greyscale compatible with $g$. Its values are either the values of $g$ (Step~\ref{step:compatible}) or are assigned by Step~\ref{step:assignation} to vertices belonging to $u-v$ geodesics such that $M_i=F(u,v)$. Owing to Step~\ref{step:compatible}, $f$ is compatible with $g$ and since both extreme tones are prefixed colours, its range is the interval $[0,1]$. It is necessary to check that the colour assignment by Step~\ref{step:assignation} is consistent, that is, both when a vertex in $V_c$ is again coloured by Step~\ref{step:assignation} in the $i$-iteration, and in the case of a vertex belonging to different such geodesics. First, let $w$ be a vertex with colour $f(w)$ belonging to a $u-v$ geodesic such that $M_i=F(u,v)$ (maximum value of Step~\ref{step:maximum}); in particular, $F(u,v)$ is greater than $F(u,w)$ and $F(w, v)$. This fact along with Lemma~\ref{lemma:tech} for $f$, $u$, $v$ and the subgraph induced by $V_{i-1} \cap V_c$ lead to $F(u,v)=F(u,w)=F(w,v)$ and $f(w)$ belonging to the interval of extremes $f(u)$ and $f(v)$. Whether $f(u) \leq f(v)$ (the reasoning is similar when $f(u) > f(v)$ and taking into account $F(u,v)=F(w,v)$), it holds that $$F(u,v)=F(u,w) \Rightarrow M_i=\frac{f(w)-f(u)}{d(w,u)} \Rightarrow f(w)=f(u)+d(w,u)\,M_i.$$ In other words, the value assigned to $w$ by Step~\ref{step:assignation} and its previous colour are the same.

On the other hand, let $P_{{u_{1}}-v_1}$ and $P_{u_2-v_2}$ be two geodesics such that $M_i=F(u_1,v_1)=F(u_2,v_2)$ (we assume, without loss of generality, $f(u_1)\leq f(v_1)$ and $f(u_2)\leq f(v_2)$) and let $w$ be a vertex in $V-V_c$ belonging to both $u_1v_1$ and $u_2v_2$. Suppose on the contrary that $f(u_1)+d(u_1,w)\,M_i>f(u_2)+d(u_2,w)\,M_i$ (the arguments are similar if the other inequality is assuming), and therefore
\begin{equation}
 \label{eq1-thm:01RMIGG}
f(u_1)>f(u_2)+(d(u_2,w)-d((u_1,w))\,M_i,
\end{equation}
From $M_i=F(u_1,v_1)$ it holds that
\begin{equation}
 \label{eq2-thm:01RMIGG}
f(v_1)=f(u_1)+(d(u_1,w)+d(w,v_1))\,M_i.
\end{equation}
Taking into account (\ref{eq1-thm:01RMIGG}) and (\ref{eq2-thm:01RMIGG}),
$$f(v_1)>f(u_2)+(d(u_2,w)-d((u_1,w))\,M_i+(d(u_1,w)+d(w,v_1))\,M_i=$$
$$=f(u_2)+(d(u_2,w)+d(w,v_1))\,M_i \geq f(u_2)+d(u_2,v_1)\,M_i \Rightarrow$$
$$\Rightarrow \frac{f(v_1)-f(u_2)}{d(u_2,v_1)} \geq M_i,$$
which is a contradiction.

The gradation vector of $f$ is now proved to be $$(M_1,\ldots,M_1,M_2,\ldots,M_2,\ldots,M_r,\ldots,M_r,0,\ldots,0),$$ where $r$ is the number of executions of the while-loop, and the number of zeros can be null. For every edge of $G$ either its vertices belong to the set $A$ at only one execution of Step~\ref{step:maxgeo} or its colour is white (extreme tone $0$) due to the situation described in Step~\ref{step:terminals}. In the first case, its colour is $M_i$ by Corollary~\ref{cor:maxwithe} applied to each pair of vertices in $S_i$ and every connected component of $G_i$ at which the maximum value $M_i$ is reached. Furthermore, it is necessary to guarantee that the sequence of maximum values computed in Step~\ref{step:maximum} is strictly decreasing in $i$. The value $d_i$ denotes the distance measured in $G_i$ and is listed in the matrix $D_i$; it is clear that $d_i(u,v) \leq d_{i+1}(u,v)$. Let $M_i$ and $M_{i+1}$ be the maximum values of the edge-colour-increase mapping $F$ on $G_i$ and $G_{i+1}$ and computed by the iterations $i$ and $i+1$ of Step~\ref{step:maximum}, respectively. Let also $u_{i+1}$ and $v_{i+1}$ be two vertices such that $M_{i+1}=F(u_{i+1},v_{i+1})$ on $G_{i+1}$. The executions at which the vertices $u_{i+1}$ and $v_{i+1}$ are coloured determine three cases:
\begin{enumerate}
\item  Both vertices $u_{i+1}$ and $v_{i+1}$ are coloured before the $i$-iteration takes place. Then, if $\{u_{i+1},v_{i+1} \} \notin S_i$, it follows that $$M_i > \frac{|f(u_{i+1})-f(v_{i+1})|}{d_{i}(u_{i+1},v_{i+1})}\geq \frac{|f(u_{i+1})-f(v_{i+1})|}{d_{i+1}(u_{i+1},v_{i+1})}=M_{i+1}.$$
    Otherwise, the fact that $\{u_{i+1},v_{i+1} \} \in S_i$ leads to $d_i(u_{i+1},v_{i+1}) < d_{i+1}(u_{i+1},v_{i+1})$ and then
    $$M_i = \frac{|f(u_{i+1})-f(v_{i+1})|}{d_{i}(u_{i+1},v_{i+1})} >  \frac{|f(u_{i+1})-f(v_{i+1})|}{d_{i+1}(u_{i+1},v_{i+1})}=M_{i+1}.$$

\item One of the vertices $u_{i+1}$ and $v_{i+1}$ is coloured at the $i$-iteration but the other one is previously. Without loss of generality, $f(u_{i+1}) \leq f(v_{i+1})$ can be assumed and then we distinguish two possibilities depending on whether either $u_{i+1}$ receives its grey tone at the $i$-iteration or $v_{i+1}$ does.

    \begin{enumerate}

    \item If the vertex $u_{i+1}$ is coloured at the $i$-iteration, it belongs to some $u_i-v_i$ geodesic such that $M_i=F(u_i,v_i)$ on $G_i$ (the inequality $f(u_i) \leq f(v_i)$ can be assumed) and so it holds that
    \begin{equation}\label{eq3-thm:01RMIGG}   f(u_{i+1})=f(u_i)+d_i(u_i,u_{i+1})\,M_i \Rightarrow f(u_i)=f(u_{i+1})-d_i(u_i,u_{i+1})\,M_i. \end{equation}
    On the other hand, $v_{i+1}$ is coloured previously to the $i$-iteration and the inequality between $M_i$ and $M_{i+1}$ is achieved by distinguishing if $\{u_i,v_{i+1}\}$ belongs or not to $S_i$. In case that $\{u_i,v_{i+1}\} \notin S_i$, the equality in (\ref{eq3-thm:01RMIGG}) leads to:
    $$M_i > \frac{f(v_{i+1})-f(u_{i})}{d_{i}(u_{i},v_{i+1})}\geq \frac{f(v_{i+1})-f(u_{i})}{d_{i}(u_{i},u_{i+1})+d_i(u_{i+1},v_{i+1})} \geq $$ $$ \geq \frac{f(v_{i+1})-f(u_{i})}{d_{i}(u_{i},u_{i+1})+d_{i+1}(u_{i+1},v_{i+1})}= $$ $$= \frac{f(v_{i+1})-f(u_{i+1})+d_i(u_i,u_{i+1})\, M_i}{d_{i}(u_{i},u_{i+1})+d_{i+1}(u_{i+1},v_{i+1})} = $$ $$= \frac{d_{i+1}(u_{i+1},v_{i+1})\, M_{i+1}+d_i(u_i,u_{i+1})\, M_i}{d_{i}(u_{i},u_{i+1})+d_{i+1}(u_{i+1},v_{i+1})}.  $$ If $M_i \leq M_{i+1}$, then $$M_i > \frac{[d_{i+1}(u_{i+1},v_{i+1})+d_i(u_i,u_{i+1})]\, M_i}{d_{i}(u_{i},u_{i+1})+d_{i+1}(u_{i+1},v_{i+1})} = M_i,$$ which is a contradiction, and therefore $M_i > M_{i+1}$.

In case that $\{u_i,v_{i+1}\} \in S_i$ and since $d_i(u_i,v_{i+1}) < d_i(u_i,u_{i+1})+d_{i+1}(u_{i+1},v_{i+1}) $ due to the connection of $u_{i+1}$ and $v_{i+1}$ in $G_{i+1}$, it follows that
$$M_i = \frac{f(v_{i+1})-f(u_i)}{d_i(u_i,v_{i+1})} > \frac{f(v_{i+1})-f(u_i)}{d_i(u_i,u_{i+1})+d_{i+1}(u_{i+1},v_{i+1})},$$ and the reasoning goes on as in the lines above.

\item If the vertex $v_{i+1}$ is coloured at the $i$-iteration, similar arguments lead to the result taking into account two facts:  $v_{i+1}$ belongs to some $u_i-v_i$ geodesic such that $M_i=F(u_i,v_i)$ on $G_i$ and $f(u_i) \leq f(v_i)$, which implies that $f(v_{i+1})=f(v_i)-d_i(v_i,v_{i+1})\, M_i$, and the membership or not of $\{u_{i+1},v_i\}$  in $S_i$.
    \end{enumerate}

\item Both vertices $u_{i+1}$ and $v_{i+1}$ are coloured by the $i$-iteration. Therefore, there exist two pairs of vertices $\{ u_i^1, v_i^1\}$ and $\{ u_i^2, v_i^2\}$ of $G_i$ such that $M_i=F(u_i^1,v_i^1)=F(u_i^2,v_i^2)$ and the vertices $u_{i+1}$ and $v_{i+1}$ belong to some $u_i^1-v_i^1$ and $u_i^2-v_i^2$ geodesic, respectively (without loss of generality we may suppose that $f(u_{i+1})\leq f(v_{i+1})$, $f(u_i^1)\leq f(v_i^1)$ and $f(u_i^2)\leq f(v_i^2)$). Then,
    \begin{equation}\label{eq4-thm:01RMIGG}f(u_{i+1})=f(u_i^1)+d_i(u_i^1,u_{i+1})\, M_{i} \Rightarrow f(u_i^1)=f(u_{i+1})-d_i(u_i^1,u_{i+1})\, M_{i}\end{equation}
    \begin{equation}\label{eq5-thm:01RMIGG}f(v_{i+1})=f(v_i^2)-d_i(v_i^2,v_{i+1})\, M_{i} \Rightarrow f(v_i^2)=f(v_{i+1})+d_i(v_i^2,v_{i+1})\, M_{i}\end{equation}

Whether $\{ u_i^1, v_i^2\} \notin S_i$ and from (\ref{eq4-thm:01RMIGG}) and (\ref{eq5-thm:01RMIGG}), it holds that
$$M_i > \frac{f(v_i^2)-f(u_i^1)}{d_i(u_i^1,v_i^2)} \geq \frac{f(v_i^2)-f(u_i^1)}{d_i(u_i^1,u_{i+1})+d_{i+1}(u_{i+1},v_{i+1})+d_i(v_{i+1},v_i^2)}=$$
$$=\frac{f(v_{i+1})+d_i(v_i^2,v_{i+1})\, M_{i}-f(u_{i+1})+d_i(u_i^1,u_{i+1})\, M_{i}}{d_i(u_i^1,u_{i+1})+d_{i+1}(u_{i+1},v_{i+1})+d_i(v_{i+1},v_i^2)}=$$
$$=\frac{d_{i+1}(u_{i+1},v_{i+1})\, M_{i+1}+d_i(v_i^2,v_{i+1})\, M_{i}+d_i(u_i^1,u_{i+1})\, M_{i}}{d_i(u_i^1,u_{i+1})+d_{i+1}(u_{i+1},v_{i+1})+d_i(v_{i+1},v_i^2)}$$ If $M_i \leq M_{i+1}$, then $$M_i > \frac{[d_i(u_i^1,u_{i+1})+d_{i+1}(u_{i+1},v_{i+1})+d_i(v_{i+1},v_i^2)]\, M_i}{d_i(u_i^1,u_{i+1})+d_{i+1}(u_{i+1},v_{i+1})+d_i(v_{i+1},v_i^2)} = M_i,$$ which is a contradiction, and hence $M_i > M_{i+1}$.

Whether $\{ u_i^1, v_i^2\} \in S_i$, the inequality between $M_i$ and $M_{i+1}$ is achieved by applying that $d_i(u_i^1,v_i^2) < d_i(u_i^1,u_{i+1})+d_{i+1}(u_{i+1},v_{i+1})+d_i(v_{i+1},v_i^2)$, which follows from the connection of $u_{i+1}$ and $v_{i+1}$ in $G_{i+1}$.
\end{enumerate}

Our next and final aim is to prove that $f$ is the only greyscale compatible with $g$ such that its gradation vector $$grad(G,f)=(M_1,\ldots,M_1,M_2,\ldots,M_2,\ldots,M_r,\ldots,M_r,0,\ldots,0)$$ is minimum among all gradation vectors of such greyscales. For this purpose, let $C_k$ be the vertex-set of $G$ containing the vertices that have been coloured at any of the first $k$ executions of the while-loop in Step~\ref{step:while-loop}, for $k=1\ldots r$, being $r$ the total number of executions of Step~\ref{step:while-loop}. Given a minimum gradation greyscale $h$ compatible with $g$ we prove by induction on $k$ that $h(w)=f(w)$ for all $w \in V$.

For $k=1$, every vertex of $C_1$ belongs to some $u-v$ geodesic $P_{u-v}=\{u=w_0,e_1,w_1,e_2,w_2, \dots, w_{l-1},e_l,v=w_l\}$ (alternating sequence of vertices and edges) such that $g(u)=f(u)=h(u)$ and $g(v)=f(v)=h(v)$. There does not exist an edge $e_i$ such that $\widehat{h}(e_i) > M_1$ due to the minimality of ${\mathcal{G}}_h$ and the existence of ${\mathcal{G}}_f$, and the next argument guarantees the non-existence of an edge $e_i$ such that $\widehat{h}(e_i) < M_1$, therefore $\widehat{h}(e_i)=\widehat{f}(e_i)=M_1$ for all edge of $P_{u-v}$. Both values of any greyscale on consecutive vertices of $P_{u-v}$ define intervals whose union is a cover of the interval $[g(u),g(v)]$ or $[g(v),g(u)]$. The lengths of these intervals, that is, the grey tones of the corresponding edges, are all $M_1$ for $f$, and so, also for the greyscale $h$. Thus, if there exists one of them less than $M_1$ there must exists another one greater than  $M_1$, but this fact is not possible owing to the minimality of ${\mathcal{G}}_h$ and the existence of ${\mathcal{G}}_f$.

Then, $f(w_i)=h(w_i)$ for all vertices of $P_{u-v}$ since $\widehat{f}(e_i)=\widehat{h}(e_i)=M_1$ for all edges of $P_{u-v}$ and $f(u)=h(u)=g(u)$ and $f(v)=h(v)=g(v)$.

For the induction step, the same previous reasoning is applied to the elements of the geodesics taking part in the execution $k+1$ of  Step~\ref{step:while-loop}, since the extreme vertices of such geodesics belong to $C_k$ and therefore their grey tones assigned by $f$ and $h$ are equal.
\end{proof}

The next result solves the \textsc{rmigg} problem in the case of only one type of extreme colour, either white or black, among the prefixed values.

\begin{thm}
\label{thm:0or1RMIGG}
Let $g$ be an incomplete $V_c$-greyscale of a connected graph $G(V,E)$ such that only one extreme tone, either $0$ or $1$, is reached by $g$. Then it is possible to obtain, in polynomial time, all the greyscales of $G$ compatible with $g$ whose gradation vector is the minimum one among all gradation vectors of greyscales compatible with $g$.
\end{thm}

\begin{proof}
Without loss of generality, we may suppose that $g$ reaches the white colour $0$ but not the black colour $1$. There exists a vertex $w \in V-V_c$ such that $f(w)=1$ for any greyscale $f$ compatible with $g$, in particular, any minimum gradation greyscale. Then, for every vertex $w \in V-V_c$ a new incomplete greyscale $g_w$ is defined such that $g_w(u)=g(u)$ whether  $u\in V_c$ and $g_w(w)=1$. In accordance with Theorem~\ref{thm:01RMIGG} for $g_w$, there exists only one greyscale $f_w$ compatible with $g_w$ whose gradation vector is minimum among all gradation vectors of greyscales compatible with $g_w$.

Among these $|V-V_c|$ greyscales $f_w$, those whose gradation vector is minimum are the solutions of the \textsc{rmigg} problem for only one extreme colour, and they have been obtained by running the \textsc{Compatible-Complete-Mapping} polynomial procedure  $|V-V_c|$ times. Hence $O(n^5)$ is achieved for this problem.
\end{proof}

Now, the following result resolves the \textsc{rmigg} problem in case of neither the black colour nor the white one appears among the prefixed values.

\begin{thm}
\label{thm:n0nor1RMIGG}
Let $g$ be an incomplete $V_c$-greyscale of a connected graph $G(V,E)$ such that no extreme tone is reached by $g$. Then it is possible to obtain, in polynomial time, a greyscale of $G$ compatible with $g$ whose gradation vector is the minimum one among all gradation vectors of greyscales compatible with $g$.
\end{thm}

\begin{proof}
This proof is similar to the proof of Theorem~\ref{thm:0or1RMIGG}, but in this case a new incomplete $V_c \cup \{w_1,w_2\}$-greyscale is defined for every pair of vertices $w_1$ and $w_2$ of $|V-V_c|$ such that $g_{\{w_1,w_2\}}(u)=g(u)$ whether  $u\in V_c$, $g_{\{w_1,w_2\}}(w_1)=0$ and $g_{\{w_1,w_2\}}(w_2)=1$.
Since there are $\binom{|V-V_c|}{2}$ of these incomplete greyscales and the \textsc{Compatible-Complete-Mapping} polynomial procedure provides only one greyscale for each one of them, this problem can be solved in $O(n^6)$ time.
\end{proof}

Finally, the following and last theorem solves the {\sc migg} problem.

\begin{thm}
\label{thm:main}
The {\sc migg} problem can be solved in polynomial time, that is, the minimum gradation vector and all their minimum gradation greyscales are obtained in polynomial time.
\end{thm}

\begin{proof}
All possible greyscales can be considered taking all possible pairs of vertices to be coloured with the extreme tones, that is, $\binom{|V|}{2}$ incomplete greyscales. Theorem~\ref{thm:01RMIGG} is applied to each of these incomplete greyscales and hence, the best greyscales in the sense of gradation have to be selected among a set of $\binom{|V|}{2}$ greyscales, each of them computed in polynomial time. Thus, the minimum gradation vector of the given graph, as well as all their minimum gradation greyscales are achieved in polynomial time.
\end{proof}

In accordance with Corollary~\ref{cor:antizo} it is possible to reduce the actual time of the \textsc{$V_c$-Compatible-Complete-Mapping} procedure applied to the {\sc migg} problem ($V_c=\emptyset$) only taking into account the pairs of antipodal vertices to be coloured with the extreme tones, instead of all the pairs of vertices of the graph. By other hand, the following observation also reduces the actual time of the algorithms that solve  the {\sc migg} problem and the \textsc{rmigg} problems when both extreme tones are not reached by the incomplete greyscale. The \textsc{$V_c$-Compatible-Complete-Mapping} procedure has to be applied a quadratic number of times in the worst case, once for each pair of vertices coloured with black and white. These executions can be performed in parallel and since we are dealing with minimax problems, after each iteration of the while-loop, it suffices to continue with the executions that lead to the minimum value for $M_i$ and, moreover, appearing the minimum number of times in the gradation vector. The rest of these executions of the \textsc{$V_c$-Compatible-Complete-Mapping} procedure can be discarded.

Finally, it is easy to check the following result.

\begin{cor}
At least, the first $d(G)$ components of the minimum gradation vector of the \textsc{migg} problem for a graph $G$ with diameter $d(G)$ are $\frac{1}{d(G)}$.
\end{cor}

\section{Open problems}
\label{sec:openp}

The new concept of gradation of a graph, related to vertex and edge colourings, has been introduced, and polynomial algorithms have been designed to solve the problem of determining the minimum gradation vector depending on whether or not there exist prefixed colours. Nevertheless, the algorithms developed in this paper have high time complexity so that the main open problem that immediately raises by our work is to improve the computational time required to solve gradation problems. Our time complexities are determined by the computation of the distance matrix of the graph so different resolution techniques would have to be investigated in order to reduce the computational times.

It would be also interesting to pose gradation in digraphs, studying the more suitable way of assigning colours to the directed edges.

\vspace*{1cm}
\noindent{\bf Acknowledgements}

\vspace*{0.5cm}
The authors gratefully acknowledge financial support by the Spanish Mi\-nis\-te\-rio de Economía, Industria y Competitividad and Junta de Andaluc\'ia via grants,  MTM2015-65397-P (M.T. Villar-Li\~n\'an) and  PAI FQM-164, respectively.


\begin{thebibliography}{9}

\bibitem{ahkms-mp-07} Aardal, Karen I., Stan P. M. van Hoesel, Arie M. C. A. Koster, Carlo Mannino, and Antonio Sassano. Models and solution techniques for frequency assignment problems. \emph{Annals of Operations Research}, 153(1):79, 2007.

\bibitem{cgrv-cgg-17} Castro, Natalia de, Mar\'{\i}a A. Garrido-Vizuete, Rafael Robles and Mar\'{\i}a Trinidad Villar. Contrast in greyscales of graphs. \textit{arXiv, Cornell University e-print repository}, paper no. arXiv:1612.07527v3, January 2018.

\bibitem{cacchm-rac-81} Chaitin, Gregory J., Marc A. Auslander, Ashok K. Chandra, John Cocke, Martin E. Hopkins, and Peter W. Markstein. Register allocation via coloring.
\emph{Computer Languages}, 6(1):47-57, 1981.


\bibitem{clw-odmwi-92} Cheng, Wu-Tung, James L. Lewandowski, and Eleanor Wu. Optimal diagnostic methods for wiring interconnects. \textit{IEEE Transactions on Computer-Aided Design of Integrated Circuits and Systems,}, 11(9):1161-1166, 1992.

\bibitem{DiNardo2013} Di Nardo, Armando, Michele Di Natale, Giovanni F. Santonastaso, Velitchko G. Tzatchkov, and Victor H. Alcocer-Yamanaka.
Water network sectorization based on graph theory and energy  performance indices.  \emph{Journal of Water Resources Planning and Management}, 140(4), 2014. doi:10.1061/(ASCE)WR.1943-5452.0000364.


\bibitem{deg-ccggcse-07} Dillencourt, Michael B., David Eppstein, and Michael T. Goodrich.
Choosing colors for geometric graphs via color space embeddings. In \emph{Michael Kaufmann and Dorothea Wagner, editors, Proc. 14th International Symposium on Graph Drawing (GD'06)}, volume 4372 of \emph{Lecture Notes in Computer Science}, 294-305, Springer, 2007.


\bibitem{f-a97sp-62} Floyd, Robert W. Algorithm 97: shortest path. \emph{Communications of the ACM}, 5:345, 1962.


\bibitem{fl-cvdascp-04}  Floudas, Christodoulos A. and Xiaoxia Lin. Continuous-time versus discrete-time approaches for scheduling of chemical processes: a review. \emph{Computers \& Chemical Engineering}, 28(11):2109-2129, 2004.


\bibitem{Girvan} Girvan, M., and M. E. J. Newman, Community structure in social and
biological networks.  {\em Proc. Natl. Acad.	Sci. USA}, 99(12):7821-7826, 2002.


\bibitem{gk-glvws-09} Griggs, Jerrold R., and Daniel Kr\'al. Graph labellings with variable weights, a survey. \emph{Discrete Applied Mathematics}, 157(12):2646-2658, 2009.



\bibitem{h-gt-90} Harary, Frank. \emph{Graph theory}.
Addison-Wesley, 1969.


\bibitem{hkv-mdgc-11} Hu, Yifan, Stephen Kobourov, and Sankar Veeramoni.
On maximum differential graph coloring. In \emph{Ulrik Brandes and Sabine Cornelsen, editors,
 Proc. 18th International Symposium on Graph Drawing (GD'10)}, volume 6502 of \emph{Lecture Notes in Computer Science}, 274-286, Springer, 2011.



\bibitem{k-rcp-72} Karp, Richard M. Reducibility among combinatorial problems. In
\emph{Raymond E. Miller and James W. Thatcher, editors, Complexity of Computer Computations}, 85-103, Springer US, 1972.



\bibitem{lvw-svbmp-84} Leung, Joseph Y.T., Oliver Vornberger and James D. Witthoff. On some variants of the bandwidth minimization problem. \emph{SIAM Journal on Computing}, 13(3):650-667, 1984.

\bibitem{lll-mscp-17} Lecat, Cl\'ement, Corinne Lucet, and Chu-Min Li. Minimum sum coloring problem: Upper bounds for the
chromatic strength. \emph{Discrete Applied Mathematics}, 233:71-82, 2017.

\bibitem{llsh-fn-13} Li, Feng, Jun Luo, Gaotao Shi and Ying He. FAVOR: frequency allocation for versatile occupancy of spectrum in wireless sensor networks. In \emph{Proceedings of the fourteenth ACM international symposium on Mobile ad hoc networking and computing}, 39-48, 2013.



\bibitem{m-gcpas-04} Marx, D$\acute{\mbox{a}}$niel. Graph colouring problems and their applications in scheduling. \emph{Electrical Engineering}, 48(1-2):11-16. 2004.


\bibitem{omgh-bsgc-16} Orden, David, Ivan Marsa-Maestre, Jose Manuel Gimenez-Guzman, and Enrique de la Hoz. Bounds on spectrum graph coloring. \emph{Electronic Notes in Discrete Mathematics}, 54:63-68, 2016.


\bibitem{Santo Fortunato} Santo Fortunato. Community detection in graphs. {\em Physics Reports} 486(3-5): 75-174, 2010.



\bibitem{w-sf-14} Walogóra, Grzegorz. Simulated annealing and tabu search for discrete-continuous project scheduling with discounted cash flows. \emph{RAIRO-Operations Research}, 48(1):1-24, 2014.


\bibitem{w-tbm-62} Warshall, Stephen. A theorem on boolean matrices. \emph{Journal of the ACM}, 9:11-12, 1962.


\bibitem{wsn-radpbnugca-91} Woo, Tai-Kuo, Stanley Y. Su, and Richard Newman-Wolfe. Resource allocation in a dynamically partitionable bus network using a graph coloring algorithm. In \textit{IEEE Transactions on Communications}, 39(12): 1794-1801, 1991.



\end{thebibliography}
\end{document}